\documentclass[11pt]{article}
\usepackage{amsmath}
\usepackage{amssymb}
\usepackage{amsthm}
\usepackage{enumitem}
\usepackage{mathtools}
\usepackage{setspace}
\usepackage[a4paper, margin=2.5cm]{geometry}

\theoremstyle{plain}
\newtheorem{theorem}{Theorem}[section]
\newtheorem{lemma}[theorem]{Lemma}
\newtheorem{proposition}[theorem]{Proposition}
\newtheorem{corollary}[theorem]{Corollary}

\theoremstyle{definition}

\renewenvironment{proof}[1][Proof.]{\begin{trivlist}
		\item[\hskip \labelsep {\bfseries #1}]}{\qed \end{trivlist}}

\newcommand{\gap}{\vspace{5.5pt}}

\newcommand{\R}{\mathcal{R}}
\newcommand{\Rn}{\mathcal{R}^n}

\newcommand{\Rnp}{\mathcal{R}_+^n}
\newcommand{\Rnpdown}{(\mathcal{R}_+^n)^{\downarrow}}
\newcommand{\Sn}{\mathcal{S}^n}

\newcommand{\Hn}{\mathcal{H}^n}

\newcommand{\V}{\mathcal{V}}

\newcommand{\mn}{{\cal M}_n}
\newcommand{\wtt}{\widetilde{T}}

\newcommand{\veps}{\varepsilon}
\newcommand{\E}{{\cal E}}
\newcommand{\I}{{\cal I}}
\newcommand{\Ik}{{\cal I}^{(k)}}

\newcommand{\ce}{\I\circ \E}
\newcommand{\one}{{\bf 1}}

\newcommand{\wsup}{\mbox{w-sup\,}}
\newcommand{\winf}{\mbox{w-inf\,}}
\newcommand{\diag}{\operatorname{diag}}
\newcommand{\tr}{\operatorname{tr}}

\newcommand{\abs}[1]{\left\vert #1 \right\vert}

\newcommand{\ip}[2]{\left< #1,\, #2 \right>}

\newcommand{\wprec}{\underset{w}{\prec}}

\title{\bf A pointwise weak-majorization inequality\\ for linear  maps over \\ 
Euclidean Jordan algebras}

\author{M. Seetharama Gowda\\Department of Mathematics and Statistics\\University of Maryland Baltimore County\\Baltimore, Maryland 21250, USA\\and\\
J. Jeong\\
        Applied Algebra and Optimization Research Center \\
        Sungkyunkwan University \\
        2066 Seobu-ro, Suwon 16419, Republic of Korea\\
         jjycjn@skku.edu }

\date{\today }
\begin{document}

\maketitle

\begin{abstract}
Given a linear map $T$ on a Euclidean Jordan algebra of rank $n$, we consider the set of all  nonnegative vectors $q$ in $\Rn$ with decreasing components 
that satisfy the     pointwise weak-majorization inequality 
$\lambda(|T(x)|)\underset{w}{\prec}q*\lambda(|x|)$, 
where $\lambda$ is the eigenvalue map and $*$ denotes the componentwise product in $\Rn$. With respect to the weak-majorization ordering, we show the existence of the  least 
 vector in this set.
 When $T$ is a positive map, the least vector is shown to be the  join (in the weak-majorization order) 
 of eigenvalue vectors of $T(e)$ and $T^*(e)$, where $e$ is the unit element of the algebra.  
These results are analogous to the results  of 
Bapat \cite{bapat1991}, proved in the setting of
the space of all $n\times n$ complex matrices with singular value map in place of the eigenvalue map.  
They also extend two recent results of Tao, Jeong, and Gowda \cite{tao-jeong-gowda} proved for quadratic representations and Schur product induced transformations. As an application, we provide an estimate on the norm of a general 
linear map relative to spectral norms.
\end{abstract}

\vspace{1cm}
\\


\section{Introduction}
This paper deals with a pointwise weak-majorization inequality for an arbitrary linear map on a Euclidean Jordan algebra.  Our motivation comes from several sources. 
In a 1991 paper, Bapat \cite{bapat1991} proves the following result for a linear map $T$ on the space $\mn$ of all $n\times n$ complex matrices: {\it  
There is a unique nonnegative vector $\eta(T)$ with decreasing components such that 
$$s(T(X))\underset{w}{\prec} \eta(T)*s(X)\,\,\mbox{for all}\,\,X\in \mn,$$
with the additional property that if the above inequality holds for some $q$ in place of $\eta(T)$, then $\eta(T)\underset{w}{\prec} q$.} Here, $s(X)$ 
denotes the vector of singular values of $X$ in $\mn$ written in the decreasing order,
$*$ denotes the componentwise
product in $\Rn$, and $\underset{w}{\prec}$ is the weak-majorization preordering relation  on  $\Rn$. Specializing the above result, Bapat proves that if {\it $T$ is 
positive} (meaning that it takes  positive semidefinite matrices to positive semidefinite matrices), then, 
{\it $\eta(T)$ is the  join  of the singular value vectors of $T(I)$ and $T^*(I)$ relative to the weak-majorization preordering}, where $I$ denotes the identity matrix.
\\

Subsequently, in a 1999 paper, Niezgoda \cite{niezgoda} studied a generalization
 in the setting of group majorization (Eaton triples) and described  equivalent formulations 
for positive linear maps, see Theorem 3.1 and  Examples 4.1 and 4.2 in \cite{niezgoda}.\\

Going in a different direction, in a recent paper, Tao, Jeong, and Gowda \cite{tao-jeong-gowda} proved, in the setting of Euclidean Jordan algebras, three weak-majorization inequalities. To elaborate, 
let $(\V, \circ, \langle\cdot,\cdot\rangle)$ denote a  Euclidean Jordan algebra  of rank $n$ with unit $e$ and carrying the trace inner product.
For $a\in \V$, consider the Lyapunov transformation $L_a$ and the quadratic representation $P_a$ defined, respectively, by
$$L_a(x):=a\circ x\quad\mbox{and}\quad P_a(x):=2a\circ (a\circ x)-a^2\circ x\quad (x\in \V).$$
Also, for any $n\times n$ {\it real symmetric  positive semidefinite matrix} $A$ and a fixed Jordan frame in $\V$, consider  the corresponding Schur product induced 
linear transformation $D_A:\V\rightarrow \V$ defined by
$D_A(x)=A\bullet x$. For any $x\in \V$, let $\lambda(x)$ denote the vector of eigenvalues of $x$ written in the decreasing order. In \cite{tao-jeong-gowda}, 
the following results were proved.

\begin{itemize}
\item [$\bullet$]
$\lambda(|L_a(x)|)\underset{w}{\prec} \lambda(|a|)*\lambda(|x|)\,\,\mbox{for all}\,\,x\in \V.$
\item [$\bullet$]
$\lambda(|P_a(x)|)\underset{w}{\prec} \lambda(a^2)*\lambda(|x|) \,\,\mbox{for all}\,\,x\in \V.$
\item [$\bullet$]
$\lambda(|D_A(x)|)\underset{w}{\prec} \lambda(|\mbox{diag}A|)*\lambda(|x|)\,\,\mbox{for all}\,\, x\in \V.$
\end{itemize}
 In the above, the transformations $L_a$, $P_a$, and $D_A$ are all self-adjoint and the last two are {\it positive,} that is, they keep the symmetric cone $\V_+$ of $\V$ invariant. 
Writing $T$ for any one of the these transformations, above statements could be written in a unified way:
$$\lambda(|T(x)|)\underset{w}{\prec} \lambda(|T(e)|)*\lambda(|x|) \,\,\mbox{for all}\,\,x\in \V.$$
Motivated by  the strong similarity between the results of Bapat and of Tao et al., we raise the question whether Bapat's results have   analogs in the 
setting of Euclidean Jordan algebras.
While there is a close connection between the $C^*$-algebra $\mn$ and the Euclidean Jordan algebra $\Hn$ of all $n\times n$  complex Hermitian matrices, 
lack of matrix type multiplication and associative properties in a general Euclidean Jordan algebra hinder routine/obvious generalizations of results and proofs. 
Yet, with powerful and elegant Euclidean Jordan algebra machinery we show the following results:  
{\it Given any linear map $T:\V\rightarrow \V$, there exists  a unique nonnegative vector $\eta(T)$ in $\Rn$ with decreasing components such that 
\begin{equation}\label{main inequality}
\lambda(|T(x)|)\underset{w}{\prec} \eta(T)*\lambda(|x|)\,\, \mbox{for all}\,\,x\in \V,
\end{equation} 
with the additional property that if the above inequality holds with $q$ in place of $\eta(T)$, then $\eta(T)\underset{w}{\prec} q$.
Furthermore, if $T$ is positive, then 
$\eta(T)$ is the  join of eigenvalue vectors of $T(e)$ and $T^*(e)$ in the weak-majorization preordering on $\Rn$.
In particular, when $T$ is positive and  self-adjoint, we have 
$\lambda(|T(x)|)\underset{w}{\prec} \lambda(T(e))*\lambda(|x|)$ for all $x\in \V.$}
We note that this last statement recovers the two results of Tao et al., stated for $P_a$ and $D_A$ (proved in \cite{tao-jeong-gowda} by different techniques). 
\\

Now, constructing some $q$ that satisfies the pointwise inequality $\lambda(|T(x)|)\underset{w}{\prec} q*\lambda(|x|)$ is easy: One can take a  large positive multiple of the vector of ones in $\Rn$.
Demonstrating the existence of  `least' $q$ (which implies uniqueness) and describing this $q$ for a positive map requires more and nontrivial work.  
In our analysis, three key results from Euclidean Jordan algebras are used. The first one is the {\it `Fan-Theobald-von Neumann inequality'} \cite{lim et al, baes, gowda-tao}:
\begin{equation}\label{ftvn}
\langle x,y\rangle \leq \langle \lambda(x),\lambda(y)\rangle \quad (x,y\in \V).
\end{equation}
The second one is the {\it `variational principle'} \cite{baes}:
\begin{equation}\label{variational principle}
S_k(x):=\lambda_1(x)+ \lambda_2(x) + \cdots + \lambda_k(x) = \underset{c\,\in\, \Ik}{\max} \langle x, c\rangle,
\end{equation}
where $\Ik$ denotes the set of all idempotents of rank $k$ in $\V$.
The third key result  is  a weak-majorization inequality (\cite{tao-jeong-gowda}, Lemma 4.2): If $\varepsilon\in \V$ with $\varepsilon^2=e$ and $x\in \V$, then
\begin{equation}\label{Lemma 4.1}
\lambda(|x\circ \veps|)\wprec \lambda(|x|)*\lambda(|\veps|).
\end{equation}

Now, observing the commonality of results in both $\mn$ and Euclidean Jordan algebras, one might wonder if there is a general framework/setting where a unified result could be obtained. 
We mention Eaton triples \cite{eaton-perlman} (equivalently, normal decomposition systems \cite{lewis} which includes both $\mn$ and simple Euclidean Jordan algebras) 
or more generally, Fan-Theobald-von Neumann systems \cite{gowda-ftvn} (which include Eaton triples and all  Euclidean Jordan algebras) as 
 possible candidates for such a unified approach. The work of Niezgoda \cite{niezgoda} mentioned earlier may be taken as a starting point in this regard.

Here is an outline of the paper. In Section 2, we cover some preliminaries. Section 3 deals with our main result (\ref{main inequality}) for linear maps and a consequence for positive ones.
It also includes a result on majorization: {\it There exists $q\in \Rn$  with decreasing components such that $\lambda(T(x))\prec q*\lambda(x)$ for all $x\in \V$ if and only if
$T$ is a scalar multiple of a doubly stochastic map.}
In Section 4, we describe some properties of the nonlinear map $\eta$ defined on the set of all linear maps on $\V$. As an application, we provide an estimate on the norm of a linear map relative to spectral norms. Examples are provided in Section 5 and a few open problems are mentioned in Section 6.
 
\section{Preliminaries}
Throughout this paper, we let $\Rn$ denote the real Euclidean $n$-space with the usual inner product and  $\Rnp$ denote the set of all nonnegative vectors in $\Rn$. We say that a vector 
$q=(q_1,q_2,\ldots, q_n)$ in $\Rn$ has decreasing components or that its  components are written in the decreasing order if $q_1\geq q_2\geq \cdots\geq q_n$. We let 
$$(\Rnp)^\downarrow=\big \{q=(q_1,q_2,\ldots, q_n)\in \Rnp: q_1\geq q_2\geq \cdots\geq q_n\big \}.$$ For $p,q\in \Rn$, we write $p\geq q$ if $p-q\in \Rnp$ and let $p*q$ denote 
their componentwise product.
We write $\one_k$ for  the vector in $\Rn$ with $1$s in the first $k$ slots, and $0$s elsewhere; $\one$ denotes the vector of ones in $\Rn$. 
Given $p=(p_1,p_2,\ldots, p_n)\in \Rn$, we write 
$|p|:=(|p_1|,|p_2|,\ldots, |p_n|)$ for the vector of absolute values
and  $p^\downarrow:=(p^\downarrow_1,p^\downarrow_2,\ldots, p^\downarrow_n)$ for the decreasing rearrangement of $p$; the latter is the vector obtained by 
rearranging the entries of $p$ in a decreasing order. We note the {\it Hardy-Littlewood-P\'{o}lya rearrangement inequality} $\langle p,q\rangle \leq \langle p^\downarrow,q^\downarrow\rangle$ and, as a consequence,
\begin{equation}\label{hlp}
\langle p,q\rangle \leq \langle |p|,|q|\rangle \leq \langle |p|^\downarrow,|q|^\downarrow\rangle\quad (p,q\in \Rn).
\end{equation}

Given two vectors $p$ and $q$ in $\Rn$, we say that $p$ is {\it weakly majorized}
 by $q$ and write $p\wprec q$ if $\sum_{i=1}^{k}p_i^\downarrow \leq \sum_{i=1}^{k}q_i^\downarrow$ for all indices $k$, $1\leq k\leq n$. If, in addition, 
$\sum_{i=1}^{n}p_i^\downarrow = \sum_{i=1}^{n}q_i^\downarrow$, we say that $p$ is {\it majorized} by $q$ and write $p\prec q$.   
For any $p\in \Rn$ and index $k\in \{1,2,\ldots, n\}$, $S_k(p)$ denotes the sum of $k$ largest components of $p$, that is,
$S_k(p):= \sum_{i=1}^{k}p_i^\downarrow .$ 
We will use the following result (\cite{bhatia}, Problem II.5.16) in $\Rn$:\\
\begin{equation}\label{bhatia problem}
\big [\,r\geq 0\,\, \mbox{and} \,\,  p\underset{w}{\prec} q\,\big ]\Rightarrow  r^\downarrow *p^\downarrow \underset{w}{\prec} r^\downarrow*q^\downarrow\Rightarrow \langle r^\downarrow, p^\downarrow\rangle \leq \langle r^\downarrow, q^\downarrow\rangle.
\end{equation}
\\
Consider the ordering relation on $\Rn$ induced by weak-majorization. While this relation is merely reflexive and transitive on $\Rn$, it becomes antisymmetric on 
$\Rnpdown$; thus, it is a  {\it partial order } 
on $\Rnpdown$. In this regard, the 
 following result of Bapat is useful.

\begin{proposition}\label{bapat1991} (Bapat \cite{bapat1991}, Lemma 3 and Corollary 4)
{\it \begin{itemize}
\item [(a)] Let $Q$ be a nonempty subset of $\Rnp$. Then, there is a unique $q^*\in \Rnpdown$ such that $q^*\wprec q$ for all $q\in Q$ and if $p\in \Rnp$ with $p\wprec q$ for all $q\in Q$, then $p\wprec q^*$. We write $\winf(Q):=q^*$.
\item [(b)] Suppose $S$ is a nonempty bounded subset of $\Rnp$. Then there is a unique $p^*\in \Rnpdown$ such that $s\wprec p^*$ for all $s\in S$ and if $p\in \Rnp$ with $s\wprec p$ for all $s\in S$, then $p^*\wprec p$. We write $\wsup(S):=p^*$. When $S=\{r,s\}$, the `join' of $r$ and $s$ is defined/denoted by $r\underset{w}{\vee} s:=\wsup(S)$.
\end{itemize}
}
\end{proposition}

In Item $(a)$ above, $q^*$ is constructed as follows. 
For $1\leq k\leq n$, let
$$\beta_k:=\underset{q\in Q}{\inf}S_k(q)$$ and
$$r_k:=\beta_k-\beta_{k-1},$$  where $\beta_0:=0$. Then $q^*:=(r_1,r_2,\ldots, r_n)$. Also, in Item $(b)$, for a nonempty bounded subset $S$ of $\Rnp$,
one defines $Q:=\{q: s\wprec q\,\,\mbox{for all}\,\,s\in S\}$ and $p^*:=\winf(Q).$ 
\\Note that when $r,s\in \Rnpdown$,
$$r\underset{w}{\vee} s\wprec \max\{r,s\},$$
where $\max\{r,s\}$ is the componentwise maximum of $r$ and $s$. 
\\

Throughout, we let  $(\V, \circ,\langle\cdot,\cdot\rangle)$ denote a Euclidean Jordan algebra of rank
$n$ with unit element $e$ \cite{faraut-koranyi, gowda-sznajder-tao};
the Jordan product and inner product of  elements $x$ and $y$ in $\V$ are, respectively, denoted by   $x\circ y$  and $\langle x,y\rangle$. We note (one of the defining properties of a Euclidean Jordan algebra): 
$$\langle x\circ y,z\rangle=\langle x,y\circ z\rangle\\,\, \mbox{for all}\,\,x,y,z\in \V.$$

It is well known \cite{faraut-koranyi} that any Euclidean Jordan algebra is a direct product/sum
of simple Euclidean Jordan algebras and every simple Euclidean Jordan algebra is isomorphic to one of five algebras,
three of which are the algebras of $n\times n$ real/complex/quaternion Hermitian matrices. The other two are: the algebra of $3\times 3$ octonion Hermitian matrices and the Jordan spin algebra.

According to the {\it spectral decomposition
theorem} \cite{faraut-koranyi}, every element $x\in \V$ has a decomposition
$x=x_1e_1+x_2e_2+\cdots+x_ne_n,$
where the real numbers $x_1,x_2,\ldots, x_n$ are (called) the eigenvalues of $x$ and
$\{e_1,e_2,\ldots, e_n\}$ is a Jordan frame in $\V$. (An element may have decompositions coming from different Jordan frames, but the eigenvalues remain the same.) 
The {\it trace} of $x$ is defined by
$\tr(x):=x_1+x_2+\cdots+x_n.$
It is  known that $(x,y)\mapsto \tr(x\circ y)$ defines another inner product on $\V$ that is compatible with the given Jordan product.
{\it Throughout this paper, we assume that the inner product on $\V$
is this trace inner product, that is,
$\langle x,y\rangle=\tr(x\circ y).$}\\

 The {\it rank} of an element $x$ is the number of nonzero eigenvalues of $x$.
We use the notation $x\geq 0$ ($x> 0$) when all the eigenvalues of $x$ are nonnegative (respectively, positive) and $x\geq y$ (or, $y\leq x$) when $x-y\geq 0$, etc. We let
$$\V_+:=\{x\in \V: x\geq 0\}$$
denote the 
{\it symmetric cone} of $\V$. It is known that $\V_+$ is a self-dual (closed convex) cone.\\
Given a function $\phi:\R\rightarrow \R$, the corresponding {\it L\"{o}wner map} (still denoted by $\phi$) is defined via spectral decomposition: 
\begin{center}
If $x=\sum_{i=1}^{n}x_ie_i$, then  $\phi(x):=\sum_{i=1}^{n}\phi(x_i)e_i.$
\end{center}
In particular, for  $x=\sum_{i=1}^{n}x_ie_i$, we define $|x|:=\sum_{i=1}^{n}|x_i|e_i$ and $x^+:=\sum_{i=1}^{n}x^{+}_ie_i$.
By writing $|x_i|=x_i\,\veps_i$, where $\veps_i=1$ when $x_i\geq 0$ and $\veps_i=-1$ when $x_i<0$, we see that 
 $|x|=x\circ \veps$, where
$\veps^2=e$ (in fact, $\veps:=\veps_1e_1+\veps_2e_2+\cdots+\veps_ne_n$).\\

For any $x\in \V$, $\lambda(x)$ -- called the {\it eigenvalue vector} of $x$ -- is the vector of eigenvalues of $x$ written in the decreasing order. We write
$$\lambda(x)=\Big ( \lambda_1(x),\lambda_2(x),\ldots, \lambda_n(x)\Big )$$
and note $\lambda_1(x)\geq\lambda_2(x)\geq \cdots\geq  \lambda_n(x)$. 
It is known that $\lambda:\V\rightarrow \Rn$ is  continuous \cite{baes}. \\
{\it We define weak-majorization and majorization in $\V$ by: $x\wprec y$ in $\V$ if and only if 
$\lambda(x)\wprec \lambda(y)$ in $\Rn$ and $x\prec y$ in $\V$ if and only if $\lambda(x)\prec \lambda(y)$ in $\Rn$.}
The following implication is a consequence of the well-known Hirzebruch's max-min theorem \cite{gowda-tao}:
$$x\leq y\Rightarrow \lambda(x)\leq \lambda(y).$$
In particular, we have $\lambda(x)\leq \lambda(|x|)$ for all $x\in \V$.\\

As $|\lambda(x)|^\downarrow=\lambda(|x|)$, combining (\ref{ftvn}) and (\ref{hlp}), we get the following useful inequality:
\begin{equation}\label{ftvn-hlp}
\langle x,y\rangle \leq \langle \lambda(|x|),\lambda(|y|)\rangle\quad (x,y\in \V).
\end{equation}

An element $c\in \V$ is an {\it idempotent} if $c^2=c$; it is said to be a {\it primitive idempotent} if it is nonzero and cannot be written as the sum of two other nonzero idempotents. 
By spectral decomposition theorem, corresponding to any nonzero idempotent $c$, there is a Jordan frame $\{e_1,e_2,\ldots, e_n\}$ such that 
$c=e_1+e_2+\cdots+e_k$
for some $k$, $1\leq k\leq n$. (Here, $e_1,e_2,\ldots, e_n$ are mutually orthogonal primitive idempotents); hence, the rank of such an idempotent is $k$.  Let  $\Ik$ denote  the set of all idempotents of rank $k$, $1\leq k\leq n$. 
We let 
$$\I=\bigcup_{k=1}^{n}\Ik=\mbox{Set of all nonzero idempotents}.$$
As $\V$ carries the trace inner product, the norm of any primitive idempotent is one and so
$\tr(c)=k$ for each $c\in \Ik$. As the set of all primitive idempotents is compact (\cite{faraut-koranyi}, page 78), it  follows easily that $\Ik$ is compact.  Hence, $\I$ is also compact.

In what follows, we use the letter $\veps$ for an element in $\V$ with $\veps^2=e$. Such an element will have eigenvalues $\pm 1$. We let 
$$ \E:=\{\veps\in \V: \veps^2=e\}\quad \mbox{and}\quad \I\circ \E:=\{c\circ \veps: c\in \I, \veps\in \E\}.$$

For any $z$ in $\V$ or in $\Rn$, {\it $S_k(z)$ denotes the sum of the first $k$ largest eigenvalues of $z$.} 
(Note that $\Rn$ is a Euclidean Jordan algebra in which the components of any vector are its eigenvalues, so our notation is consistent.)  \\

We let ${\cal L}(\V)$ denote the space of all (continuous) linear maps over $\V$. For $T\in {\cal L}(\V)$, we say that 
$T$ is 
\begin{itemize}
\item [(a)] {\it positive} if $T(\V_+)\subseteq \V_+$;
\item [(b)] {\it doubly substochastic} if $T$ is positive, $T(e)\leq e$ and $T^*(e)\leq e$, where $T^*$ denotes the adjoint of $T$; 
\item [(c)] {\it doubly stochastic} if $T$ is positive, $T(e)= e$ and $T^*(e)=e$; 
\item [(d)] {\it an algebra automorphism} if $T$ is invertible and $T(x\circ y)=T(x)\circ T(y)$ for all $x,y\in \V$;
\item [(e)]  {\it  a cone automorphism}, if  $T(\V_+)=\V_+$.
\end{itemize}
We respectively write 
DSS$(\V)$,  DS$(\V)$, Aut$(\V)$, and Aut$(\V_+)$ for the set of all doubly substochastic maps, doubly stochastic maps, algebra automorphisms, and cone automorphisms on $\V$.
The following results are known:
\begin{itemize}
\item [$\bullet$] $\mbox{conv(Aut}(\V))\subseteq DS(\V)$, where `conv' stands for the convex hull, see \cite{gowda-positive and ds}.
\item [$\bullet$] $T$ is doubly substochastic if and only if $\lambda(T(x))\wprec \lambda(x)$ for all $x\geq 0$ in $\V$, see \cite{jeong et al}, Theorem 3.3.
\item [$\bullet$]  $T$ is doubly stochastic if and only if $\lambda(T(x))\prec \lambda(x)$ for all $x$ in $\V$, see \cite{jeong-gowda}, Lemma 2.
\end{itemize}

\section{Results}

In this section, we establish our main result and state a consequence for positive maps. We will also consider a majorization result.

\begin{theorem}\label{main theorem}
{\it Let $T:\V\rightarrow \V$ be a linear map. Then, there exists  a unique $\eta(T)$ in $\Rnpdown$ such that
\begin{equation}\label{main inequality2}
\lambda(|T(x)|)\underset{w}{\prec} \eta(T)*\lambda(|x|)\quad \mbox{for all}\,\,x\in \V,
\end{equation}
with the additional property that if the above statement holds with $q$ in place of $\eta(T)$, then $\eta(T)\underset{w}{\prec} q$.
}
\end{theorem}

Before the proof, we cover some preliminary results. 
Let $T:\V\rightarrow \V$ be an arbitrary (but fixed) linear map. Correspondingly, we define the following sets in $\Rnpdown$:

\begin{equation}\label{Q}
Q: =\big \{q\in \Rnpdown: \lambda(|T(x)|)\underset{w}{\prec} q*\lambda(|x|)\quad \mbox{for all}\,\,x\in \V\big \},
\end{equation}
\begin{equation} \label{Lambda}
\Lambda: = \big \{\lambda(\abs{T(c \circ \varepsilon)}) :c \circ \varepsilon \in \ce\big \}, 
\end{equation}
\begin{equation}\label{Lambda*}
\Lambda^{*}: = \big \{\lambda(\abs{T^{*}(c \circ \varepsilon)}):  c \circ \varepsilon \in \ce\big \},
\end{equation}
\begin{equation}\label{S}
S:=\Lambda \cup \Lambda^*.
\end{equation}

We will show below that $Q$ is nonempty. Hence, by Proposition \ref{bapat1991}, $\winf(Q)$ is defined. Also, by the compactness of $\ce$ and the continuity of $T$ and $\lambda$, it follows that $\Lambda$ and $\Lambda^*$ are both compact. Hence, $\wsup(\Lambda$) and $\wsup(\Lambda^*)$ are defined. As $S$ is the union of $\Lambda$ and $\Lambda^*$, we have 
$$\wsup(S)=\wsup(\Lambda)\underset{w}{\vee}\wsup(\Lambda^*).$$

 In the following result, we describe the set $Q$ in different, but equivalent ways.

\begin{lemma}\label{lemma1}
{\it 
Let $T$ be a linear map on $\V$ and $q\in \Rnpdown$. 
Then, the following are equivalent:
\begin{itemize}
\item [(i)] $\lambda(|T(x)|)\underset{w}{\prec} q*\lambda(|x|)\,\, \mbox{for all}\,\,x\in \V.$
\item [(ii)] $\lambda(|T(c\circ \veps)|)\underset{w}{\prec} q*\lambda(|c\circ \veps|)\,\, \mbox{for all}\,\,c\circ \veps\in \ce.$
\item [(iii)] $\lambda(|T(c\circ \veps)|)\underset{w}{\prec} q*\lambda(|c|)\,\, \mbox{for all}\,\,c\circ \veps\in \ce.$
\item [(iv)] $\lambda(|T^*(x)|)\underset{w}{\prec} q*\lambda(|x|)\,\, \mbox{for all}\,\,x\in \V.$
\item [(v)] $\lambda(|T^*(c\circ \veps)|)\underset{w}{\prec} q*\lambda(|c\circ \veps|)\,\, \mbox{for all}\,\,c\circ \veps\in \ce$.
\item [(vi)] $\lambda(|T^*(c\circ \veps)|)\underset{w}{\prec} q*\lambda(|c|)\,\, \mbox{for all}\,\,c\circ \veps\in \ce.$
\end{itemize}
}
\end{lemma}

\begin{proof}
$(i)\Rightarrow (ii)$: This follows by specializing $x$ to  $c\circ \veps$.  \\
$(ii)\Rightarrow (iii)$: 
Suppose $(ii)$ holds so that  
$\lambda(|T(c\circ \veps)|)\underset{w}{\prec} q*\lambda(|c\circ \veps|)\,\, \mbox{for all}\,\, c\circ \veps\in \ce.$
As $\lambda(|\veps|)=\one$, in  view of (\ref{Lemma 4.1}) and (\ref{bhatia problem}), this simplifies to 
$\lambda(|T(c\circ \veps)|)\underset{w}{\prec} q*\lambda(|c|)\,\, \mbox{for all}\,\, c\circ \veps\in \ce,$ which is $(iii)$.\\
$(iii)\Rightarrow (iv)$:
We assume $(iii)$. To see $(iv)$, we have to show that for each index $k$, $1\leq k\leq n$,  and $x\in \V$, 
$$S_k(|T^*(x)|)\leq S_k(q*\lambda(|x|)).$$
Fix $k$ and $x$, and let $c\in \Ik$ be arbitrary. Then, writing $|T^*(x)|=T^*(x)\circ \varepsilon$ for some $\varepsilon \in \E$, we have 
$$
\begin{array}{lcl}
\langle |T^*(x)|,c\rangle &=&\langle T^*(x)\circ \varepsilon, c\rangle\\
                          &=&\langle T^*(x), c\circ \varepsilon\rangle\\
                          &=&\langle x,T(c\circ \varepsilon)\rangle\\
                          &\leq&  \langle \lambda(|x|),\lambda(|T(c\circ\varepsilon)|)\rangle\\
                          &\leq & \langle \lambda(|x|), q*\lambda(|c|)\rangle \\
                          &=&  \langle q*\lambda(|x|), \lambda(|c|)\rangle,
   \end{array}
$$
where the first inequality is due to (\ref{ftvn-hlp}),  
and the second one is due to condition $(iii)$ coupled with (\ref{bhatia problem}).   
Then, as $\lambda(|c|)=\lambda(c)=\one_k$, we have
$$\langle |T^*(x)|,c\rangle\leq \langle q*\lambda(|x|), \lambda(c)\rangle =\sum_{i=1}^{k} (q*\lambda(|x|))_i=S_k(q*\lambda(|x|)).$$
Now, taking the maximum over $c\in \I^{(k)}$ and using (\ref{variational principle}), we get $S_k(|T^*(x)|)\leq S_k(q*\lambda(|x|).$ This proves that 
$$\lambda(|T^*(x)|)\underset{w}{\prec} q*\lambda(|x|)\,\, \mbox{for all}\,\,x\in \V.$$\\
$(iv)\Rightarrow (v)$: This can be seen by specializing $x$ to $c\circ \veps$.\\
The implications $(v)\Rightarrow (vi)$ and $(vi)\Rightarrow (i)$ are seen by replacing  $T$ by $T^*$ in the implications $(ii)\Rightarrow (iii)$ and $(iii)\Rightarrow (iv)$. 
\end{proof}

\gap

\noindent{\bf Remark 1.} Recall that a linear map $T$ is {\it positive} if $T(\V_+)\subseteq \V_+$. For such a map, it is known (\cite{tao-jeong-gowda}, Example 3.7) that 
\begin{equation}\label{example 3.7}
\lambda(|T(x)|)\wprec \lambda(T(|x|))\,\,\mbox{for all}\,\, x\in \V.
\end{equation}
Using this and with appropriate modifications, 
we can simplify Lemma \ref{lemma1} and its proof as follows (stated without proof):\\

{\it When $T$ is positive and $q\in \Rnpdown$, the following are equivalent: 
\begin{itemize}
\item [(i)] $\lambda(|T(x)|)\underset{w}{\prec} q*\lambda(|x|)\,\, \mbox{for all}\,\,x\in \V.$
\item [(ii)] $\lambda(T(x))\underset{w}{\prec} q*\lambda(x)\,\, \mbox{for all}\,\,x\geq 0.$
\item [(iii)] $\lambda(T(c))\underset{w}{\prec} q*\lambda(c)\,\, \mbox{for all}\,\,c\in \I.$
\item [(iv)] $\lambda(T^*(x))\underset{w}{\prec} q*\lambda(x)\,\, \mbox{for all}\,\,x\geq 0.$
\item [(v)] $\lambda(T^*(c))\underset{w}{\prec} q*\lambda(c)\,\, \mbox{for all}\,\,c\in {\cal I}.$
\end{itemize}
}
In \cite{niezgoda}, Theorem 3.1, Niezgoda proves a result for certain types of linear  maps in the setting of Eaton triples. When specialized, 
it will yield a result of the above type for positive linear maps on {\it simple} Euclidean Jordan algebras.
(Note: So far, it is only known that every simple Euclidean Jordan algebra is a normal decomposition system, equivalently, an Eaton triple \cite{lim et al}.) 
\gap

\begin{lemma} \label{lemma2}
Given a linear map $T:\V\rightarrow \V$, consider the sets $Q$ and $S$ defined in (\ref{Q})-(\ref{S}). Then the following hold:
\begin{itemize}
\item [(i)] $s\underset{w}{\prec} q$ for all $s\in S$ and $q\in Q$.
\item [(ii)] $\wsup(S)\in Q$.
\item [(iii)] $\wsup(S)=\winf(Q).$
\end{itemize}
\end{lemma}

\begin{proof}
We first observe that $Q$ is nonempty. This can be seen by taking $q=t\,\one$, where $t$ is a large positive number, and using Item $(iii)$ in Lemma \ref{lemma1} along with the compactness of $\Lambda$.\\ 
$(i)$ Let $q\in Q$ so that (by the previous lemma), 
$\lambda(\abs{T(c \circ \varepsilon)}) \underset{w}{\prec} q \ast \lambda(c)$  and $\lambda(\abs{T^*(c \circ \varepsilon)}) \underset{w}{\prec} q \ast \lambda(c)$ 
for all $c\circ \veps\in \ce$. Let $s\in S=\Lambda \cup \Lambda^*$. If 
$s\in \Lambda$, then $s=\lambda(\abs{T(c \circ \varepsilon)})$ for some $c\circ \veps\in \ce$, in which case,
$$\lambda(\abs{T(c \circ \varepsilon)}) \underset{w}{\prec} q \ast \lambda(c)\wprec q,$$ 
where the second inequality follows from the nonnegativity of $q$ and the fact that $\lambda(c)=\one_k$ for some index $k$.
A similar statement ensues if $s\in \Lambda^*$. Hence, $s\wprec q$. This proves $(i)$. We also see, by Proposition \ref{bapat1991} that   
$$\wsup(S)\wprec \winf(Q).$$
$(ii)$ Let  $q:=\wsup(S)$. 
To see $q\in Q$, it is enough to show by Lemma \ref{lemma1}, 
        \begin{equation} \label{eq: weak majorization inequality2}
                \lambda(\abs{T(c \circ \varepsilon)}) \underset{w}{\prec} q \ast \lambda(c)
        \end{equation}
        for all $c\circ \veps\in \ce$. Consider an index $k$ and  $c \in \Ik$. Then \eqref{eq: weak majorization inequality2} is equivalent to
        \begin{equation} \label{eq: weak majorization inequality3}
                S_l(\abs{T(c \circ \varepsilon)}) \leq S_{\min\{k,\, l\}}(q)\quad\mbox{for all}\,\,1\leq l\leq n.
        \end{equation}
        Now, fix $c \circ \varepsilon\in \ce$,  $l\in \{1,2,\ldots,n\}$, and  choose $\varepsilon' \in \E$ such that  $\abs{T(c \circ \varepsilon)} = T(c \circ \varepsilon) \circ \varepsilon'$, and let 
 $c' \in \I^{(l)}$. If $l \leq k$, then we have
        \begin{align*}
                \ip{\abs{T(c \circ \varepsilon)}}{c'} & = \ip{T(c \circ \varepsilon) \circ \varepsilon'}{c'} \\
                & \leq \ip{\lambda(\abs{T(c \circ \varepsilon) \circ \varepsilon'})}{\lambda(c')} \\
                & \leq \ip{\lambda(\abs{T(c \circ \varepsilon)})}{\lambda(c')} \\
                & \leq \ip{q}{\lambda(c')} \\
                & = S_l(q),
        \end{align*}
where the first inequality is due to (\ref{ftvn-hlp}),  
and the second  inequality is due to (\ref{Lemma 4.1}) coupled with (\ref{bhatia problem}). The last inequality is due to $\lambda(\abs{T(c \circ \varepsilon)})\wprec q$ as $q=\wsup(S)$. By taking the supremum over $c' \in \I^{(l)}$, we get $S_l(\abs{T(c \circ \varepsilon)}) \leq S_l(q)$.

 On the other hand, if $l > k$, then
        \begin{align*}
                \ip{\abs{T(c \circ \varepsilon)}}{c'} & = \ip{T(c \circ \varepsilon) \circ \varepsilon'}{c'} \\
                & = \ip{T(c \circ \varepsilon)}{c' \circ \varepsilon'} \\
                & = \ip{c \circ \varepsilon}{T^*(c' \circ \varepsilon')} \\
                & = \ip{c}{[T^*(c' \circ \varepsilon')] \circ \varepsilon} \\
                & \leq \ip{\lambda(c)}{\lambda(\abs{[T^*(c' \circ \varepsilon')] \circ \varepsilon})} \\
                & \leq \ip{\lambda(c)}{\lambda(\abs{T^*(c' \circ \varepsilon')})} \\
                & \leq \ip{\lambda(c)}{q} \\
                & = S_k(q),
        \end{align*}
where the first inequality is due to (\ref{ftvn-hlp}), second one due to (\ref{Lemma 4.1}), and the last one is to  the inequality $\lambda(\abs{T^*(c' \circ \varepsilon')})\wprec q=\wsup (S)$.  Again, taking the supremum over $c' \in \I^{(l)}$ we get
 $S_l(\abs{T(c \circ \varepsilon)}) \leq S_k(q)$. Hence, we have proved (\ref{eq: weak majorization inequality3}), so $q=\wsup(S)\in Q$. 
This proves $(ii)$.\\
$(iii)$ From $(ii)$, $q:=\wsup(S)\in Q$, hence, $\winf(Q)\wprec q= \wsup(S)$. From $(i)$, the reverse inequality holds. Since the weak-majorization ordering is antisymmetric on $\Rnpdown$, we have $(iii)$.
\end{proof}

We now come to the proof of our main theorem.

\gap 

\noindent{\bf Proof of Theorem \ref{main theorem}.}
Given $T$, we define $Q$ and $S$ as  (\ref{Q})-(\ref{S}). Let $\eta(T):=\wsup(S)$. As this belongs to $Q$, we see statement (\ref{main inequality2}) in Theorem \ref{main theorem}. The additional item
 follows from the equality $\eta(T)=\wsup(S)=\winf(Q)$.
$\hfill$ $\qed$

\gap

Generally, finding/describing  $\eta(T)$ may not be easy. However, when $T$ is a positive map, we have a simple expression for $\eta(T)$.

\begin{corollary}\label{corollary for positive maps}
Let $T$  be a positive linear map on $\V$. Then, $\wsup(\Lambda)= \lambda(T(e))$ and $\wsup(\Lambda^*)= \lambda(T^*(e))$. 
Hence,
$$\eta(T)=\lambda\big (T(e)\big )\underset{w}{\vee}\lambda\big (T^*(e)\big ).$$
In particular, if $T$ is also self-adjoint, then $\eta(T)=\lambda\big (T(e)\big )$.
\end{corollary}

\begin{proof}  Consider $c\circ \veps\in \ce$. Since $\lambda(\abs{c \circ \varepsilon}) \underset{w}{\prec} \lambda(c) \leq \lambda(e)=\one$, the first component in $\lambda(\abs{c \circ \varepsilon})$ is less than or equal to $1$. Hence all components in $\lambda(\abs{c \circ \varepsilon})$ are less than or equal to one. 
It follows (by considering the spectral decomposition) that $\abs{c \circ \varepsilon} \leq e$. 
As $T$ is positive, $T(\abs{c \circ \varepsilon})\leq T(e)$ and so $\lambda(T(\abs{c \circ \varepsilon}))\leq \lambda( T(e))$. However, by (\ref{example 3.7}),
$$\lambda(\abs{T(c\circ \veps)}) \underset{w}{\prec} \lambda(T(\abs{c \circ \varepsilon})).$$
Hence, $\lambda(\abs{T(c\circ \veps)}) \underset{w}{\prec}\lambda(T(e)).$ As $c\circ \veps$ is arbitrary in $\ce$, we see that 
$$\wsup(\Lambda)\wprec \lambda(T(e)).$$
But $e\in \ce$, and so, $\lambda(T(e))\wprec \wsup(\Lambda).$ Thus, 
$$\wsup(\Lambda)= \lambda(T(e)).$$ 
Now, $T^*$ is also positive (this is due to $\V_+$ being a self-dual cone); hence, by above, $\wsup(\Lambda^*)= \lambda(T^*(e))$. So, 
$$\eta(T)=\wsup(\Lambda\cup\Lambda^*)=\lambda(T(e))\underset{w}{\vee} \lambda(T^*(e)).$$ 
\end{proof}

\gap

We note a simple bound when $T$ is positive: 
$$\eta(T)=\lambda(T(e))\underset{w}{\vee}\lambda(T^*(e))\wprec \max\{\lambda(T(e)),\lambda(T^*(e))\}.$$

\gap

\noindent{\bf Remark 2.} In the last statement of the above corollary, the requirement that $T$ be self-adjoint can be slightly relaxed.  Suppose $T=P\,\Phi$, where $\Phi$ is an algebra automorphism of $\V$ and $P$ is 
positive and self-adjoint; see Example 4 for such a map.  As $\V$ carries the trace inner product, 
$\Phi^*$ is also an algebra automorphism of $\V$, hence preserves eigenvalues. So, $\lambda((T^*(e))=\lambda(\Phi^*P(e))=\lambda(P(e))=\lambda(P(\Phi(e)))=\lambda(T(e))$. 
Thus, when $T=P\,\Phi$, 
$$\eta(T)=\lambda(T(e)).$$ We note that if $T=\Phi\,P$, where $\Phi$ and $P$ are as above, then, $\eta(T)=\lambda(T^*(e))$.\\

\gap
 
Motivated by our  main theorem, we ask if (\ref{main inequality2}) has a majorization analog. The following result provides an answer.
\\

In what follows, we use the fact that $(-\one*\lambda(x))^\downarrow =\lambda(-x)$ for all $x\in \V$, and note that  $p\prec q$ in $\Rn$ is, by definition,  equivalent to $p^\downarrow \prec q^\downarrow$. 

\begin{theorem}
Let $T$ be a  linear map on $\V$. Then, the following are equivalent:
\begin{itemize}
\item [(i)] There exists a vector $q$ in $\Rn$ with decreasing components such that $\lambda(T(x))\prec q*\lambda(x)$ for all $x\in \V$. 
\item [(ii)] $T$ is a scalar multiple of a doubly stochastic map. 
\end{itemize}
\end{theorem}

\begin{proof}
$(i)\Rightarrow (ii)$:
We assume that $q$ in $(i)$ is given by $q=(q_1,q_2,\ldots, q_n)$. We fix a  Jordan frame $\{f_1,f_2,\ldots, f_n\}$ in $\V$ and let  
$a:=\sum_{i=1}^{n} q_if_i$ so that $q=\lambda(a)$. Then $(i)$ reads
$$\lambda(T(x))\prec \lambda(a)*\lambda(x)\,\, \mbox{for all}\,\, x\in \V.$$
This implies that  $\sum_{i=1}^{n}\lambda_i(T(x))=\sum_{i=1}^{n}\lambda_i(a)\lambda_i(x)$ for all $x$. Since $\V$ carries the trace inner product,
$\sum_{i=1}^{n}\lambda_i(T(x))=\langle T(x),e\rangle=\langle x,T^*(e)\rangle$ and so
$$\langle x, T^*(e)\rangle=\langle \lambda(x),\lambda(a)\rangle.$$
Let $b:=T^*(e)$ so that
\begin{equation}\label{first equality}
\langle x, b\rangle=\langle \lambda(x),\lambda(a)\rangle\,\,\mbox{for all}\,\,x\in \V.
\end{equation}
We claim that 
\begin{equation}\label{second equality}
\langle \lambda(x),\lambda(b)\rangle=\langle \lambda(x),\lambda(a)\rangle\,\,\mbox{for all}\,\,x\in \V.
\end{equation}
To see this, fix any $x\in \V$ and consider the spectral decomposition  $b=\sum_{i=1}^{n}\lambda_i(b)\,e_i$, where $\{e_1,e_2,\ldots, e_n\}$ is a   Jordan frame. 
Corresponding to this Jordan frame and the given $x$, define 
$$y:=\sum_{i=1}^{n}\lambda_i(x)\,e_i.$$
 Then, applying (\ref{first equality}) to $y$, we have
$$\langle y,b\rangle =\langle \lambda(y),\lambda(a)\rangle.$$ As $\lambda(x)=\lambda(y)$ and
$\langle y,b\rangle =\sum_{i=1}^{n} \lambda_i(x)\,\lambda_i(b)= \langle \lambda(x),\lambda(b)\rangle$, we see that
$$\langle \lambda(x),\lambda(b)\rangle=\langle \lambda(x),\lambda(a)\rangle.$$ This proves our claim. By specializing $\lambda(x)$ in (\ref{second equality}) (for example, to $\one_k$ for any index $k$, $1\leq k\leq n$),
we get $\lambda(b)=\lambda(a)$. But then, (\ref{first equality}) leads to  
$$\langle x, b\rangle=\langle \lambda(x),\lambda(b)\rangle\,\,\mbox{for all}\,\,x\in \V.$$
This means that every $x$ in $\V$ `strongly operator commutes' \cite{gowda-ftvn} with $b$. As shown in the Remark below, this can happen if and only if $b$ is a
scalar multiple of $e$. Since $\lambda(b)=\lambda(a)$, $a$ must also be a scalar multiple of $e$ (this can be seen via the spectral decomposition). 
Let $a=t\,e$ for some $t\in \R$ so that $q=t\one$ and 
\begin{equation}\label{multiple of ds}
\lambda(T(x))\prec t\one*\lambda(x)\,\, \mbox{for all}\,\, x\in \V.
\end{equation}
 If $t=0$, then $\lambda(T(x))\prec 0$ for all $x$. Hence, $T=0$, that is, $T$ is a multiple of the Identity map (which is doubly stochastic).   When $t$ is nonzero, we complete the proof by showing that $D:=\frac{1}{t}T$ is doubly stochastic. $t> 0$,
because of (\ref{multiple of ds}), $\lambda(D(x))\prec \one *\lambda(x)$ for all $x$; so $D$ is  doubly stochastic. When $t<0$, say $t=-1$, 
$\lambda(-1\,T(x))=\lambda(T(-x))\prec q*\lambda(-x)$ implies
$$\lambda(D(x))=\lambda(-1\,T(x))\prec \lambda(T(-x))\prec (-1\,\one*\lambda(-x))^\downarrow=\lambda(x),$$
that is,  
$\lambda(D(x))\prec \lambda(x)$ for all $x\in \V$. Now, $T=tD$ says that  $T$ is a scalar multiple of $D$.\\
$(ii)\Rightarrow (i)$. Suppose  $T=tD$, where $t\in \R$ and $D$ is a doubly stochastic map. Since
$\lambda(D(x))\prec \one *\lambda(x)$ for all $x$, by scaling,
$\lambda(T(x))\prec q*\lambda(x)$ for all $x$. (This scaling is obvious when $t\geq 0$. When $t<0$, say $t=-1$,
$\lambda(T(x))=\lambda(-D(x))=\lambda(D(-x))\prec \lambda(-x)=(-\one*\lambda(x))^\downarrow$.)
Putting $q=t\one $, we see that $\lambda(T(x))\prec q*\lambda(x)$ for all $x\in \V$.
\end{proof}

\noindent{\bf Remark 3.} We show that if $b$ {\it strongly operator commutes} \cite{gowda-ftvn} with every $x\in \V$, that is, if $\langle x,b\rangle=\langle \lambda(x),\lambda(b)\rangle$ for all $x\in \V$, then $b$ is a scalar multiple of $e$. Suppose this condition holds. Then, putting $x=-b$, we have $-||b||^2=\langle -b,b\rangle=\langle \lambda(-b),\lambda(b)\rangle.$ As $\V$ carries the trace inner product, $||\lambda(y)||^2=||y||^2$ for all $y\in \V$. Hence,
$$||\lambda(-b)+\lambda(b)||^2=\langle \lambda(-b)+\lambda(b),\lambda(-b)+\lambda(b)\rangle=||b||^2-2||b||^2+||b||^2=0.$$
This implies that $\lambda(-b)=-\lambda(b)$. As $\lambda(-b)$ has components in the decreasing order and $-\lambda(b)$ has components in the increasing order, we see that $\lambda_1(b)=\lambda_n(b)$. This proves that all components in $\lambda(b)$ are equal. Thus, $b$ is a multiple of $e$.

\section{Properties of $\eta$}
In this section, we describe some properties of the map 
$\eta:T\mapsto \eta(T)$
from ${\cal L}(\V)$ to $\Rnpdown,$
where ${\cal L}(\V)$ denotes the set of all (continuous) linear maps on $\V$. First, some notation. 
For any $x\in \V$, we let   
$$||x||_{\infty}:= ||\lambda(x)||_\infty=\underset{1\leq i\leq n}{\max}\, |\lambda_i(x)|=\lambda_1(|x|)$$ and 
for $T\in {\cal L}(\V)$,  
$$||T||_{\infty}:=\underset{0\neq x\in \V}{\sup} \frac{||T(x)||_\infty}{||x||_\infty}.$$

\gap

\begin{theorem}\label{theorem on eta}
{\it The following statements hold for $T,T_1,T_2\in {\cal L}(\V)$ and $\alpha\in \R$: 
\begin{itemize}
\item [(a)] $||\eta(T)||_\infty= \max\{||T||_\infty,||T^*||_\infty\}$. 
\item [(b)] $\eta(\alpha\,T)=|\alpha|\,\eta(T)$.
\item [(c)] $\eta(T_1T_2)\wprec \eta(T_1)*\eta(T_2)$.
\item [(d)] $\eta(T_1+T_2)\wprec \eta(T_1)+\eta(T_2)$.
\item [(e)] $\eta$ is continuous.
\item [(f)] $\eta$ is  `isotonic':  If $T_1(x)\prec T_2(x)$ for all $x\in \V$, then $\eta(T_1)\wprec \eta(T_2)$. 
\end{itemize}
}
\end{theorem}

\begin{proof}
$(a)$ Fix $T$. From the pointwise inequality, $\lambda(|T(x)|)\underset{w}{\prec} \eta(T)*\lambda(|x|)$, we see that 
$$||T(x)||_\infty=\lambda_1(|T(x)|)\leq  (\eta(T))_1\,\lambda_1(|x|)=||\eta(T)||_\infty\,||x||_\infty.$$
It follows that $||T||_{\infty}\leq ||\eta(T)||_\infty.$ By Lemma \ref{lemma1}, $\lambda(|T^*(x)|)\underset{w}{\prec} \eta(T)*\lambda(|x|)$.
This yields, $||T^*||_{\infty}\leq ||\eta(T)||_\infty.$ Hence, $$\max\{||T||_\infty,||T^*||_\infty\}\leq ||\eta(T)||_\infty.$$
 We now prove the reverse inequality.
Let, for $1\leq k\leq n$,
$$\theta_k:=\underset{c\circ \veps\in \ce}{\sup}\lambda_k(|T(c\circ \veps)|)\quad \mbox{and} \quad \theta^*_k:=\underset{c\circ \veps\in \ce}{\sup}\lambda_k(|T^*(c\circ \veps)|).$$
Let $\theta:=(\theta_1,\theta_2,\ldots,\theta_n)$ and  $\theta^*:=(\theta^*_1,\theta^*_2,\ldots,\theta^*_n)$. As the components in any $\lambda(x)$ are decreasing, we see that
$\theta,\theta^*\in \Rnpdown.$ Let 
$$\bar{q}:=\max\{\theta,\theta^*\}.$$
For the given $T$, we define the  sets $Q$, $\Lambda$, $\Lambda^*$ and $S$ as in (\ref{Q})-(\ref{S}). 
Suppose $s\in \Lambda$ so that $s=\lambda(|T(c\circ \veps)|)$ for some $c\circ \veps\in \ce$. Then, for any $k$, $1\leq k\leq n$,
$$S_k(s)=\sum_{i=1}^{k}\lambda_i(|T(c\circ \veps)|)\leq \sum_{i=1}^{k}\theta_i\leq S_k(\bar{q}).$$
We have a similar statement when $s\in \Lambda^*$. Hence, $s\wprec \bar{q}$ for all $s\in S$. This implies that $\wsup(S)\wprec\bar{q}$. However, $\wsup(S)=\winf(Q)=\eta(T)$ and so, 
$\eta(T)\wprec \bar{q}$. Thus, 
$$\eta(T)\wprec \max\{\theta,\theta^*\}.$$
This implies that 
$$||\eta(T)||_\infty=(\eta(T))_1\leq \max\{ \theta_1,\theta^*_1\}.$$
However, 
$$\theta_1=\underset{c\circ \veps\in \ce}{\sup}\lambda_1(|T(c\circ \veps)|)= \underset{c\circ \veps\in \ce}{\sup} ||T(c\circ \veps)||_\infty \leq ||T||_\infty,$$
where the  inequality is due to   
$$||T(c\circ \veps)||_\infty \leq ||T||_\infty\,||c\circ \veps||_\infty= ||T||_\infty\,\lambda_1(|c\circ \veps|)\leq ||T||_\infty\,\lambda_1(|c|)=||T||_\infty.$$ 
Similarly, $\theta^*_1\leq ||T^*||_\infty$. Hence, 
$$||\eta(T)||_\infty\leq \max\{||T||_\infty,||T^*||_\infty\}.$$
Since the reverse inequality has already been proved, we have Item $(a)$.\\
$(b)$ This is easy to see from the uniqueness part in the main theorem.\\
$(c)$ Let $T_1,T_2\in {\cal L}(\V).$ Then,
$$\lambda(|T_1T_2(x)|)\underset{w}{\prec} \eta(T_1)*\lambda(|T_2(x)|)\wprec  \eta(T_1)*\eta(T_2)*\lambda(|x|) \,\, \mbox{for all}\,\,x\in \V,$$
where we have used (\ref{bhatia problem}) in the second inequality. From the main theorem, we have $\eta(T_1T_2)\wprec \eta(T_1)*\eta(T_2)$.\\
$(d)$ Suppose $a,b\in \V$. Writing $|a+b|=(a+b)\circ \veps$  for some $\veps\in \E$, we have 
\begin{align*}
\lambda(|a+b|)& = \lambda\Big ( (a+b)\circ \veps\Big)\\
              &= \lambda\Big ( a\circ \veps+ b\circ \veps\Big )\\
              &\prec \lambda( a\circ \veps)+\lambda( b\circ \veps)\\
              &\leq \lambda( |a\circ \veps|)+\lambda( |b\circ \veps|)\\
              & \wprec \lambda( |a|)+\lambda( |b|),
\end{align*}
where the first inequality follows from Lidskii type inequality $\lambda(x+y)\prec \lambda(x)+\lambda(y)$ in $\V$ (which easily follows from a result on simple Euclidean Jordan algebras \cite{moldovan} or from a general result on hyperbolic polynomials \cite{gurvits}) and the last inequality follows from (\ref{Lemma 4.1}).
Now, we put $a=T_1(x)$, $b=T_2(x)$ and use the main theorem to get, for any $x\in \V$,
$$\lambda\Big (|T_1(x)+T_2(x)|\Big )\wprec \Big (\eta(T_1)+\eta(T_2)\Big )*\lambda(|x|).$$ 
This gives the stated conclusion in $(d)$.\\
$(e)$ From Items $(b)$ and $(d)$, we see that for each index $k$, the function 
$T\mapsto S_k(\eta(T))$ is positively homogeneous and subadditive, hence convex. As convex functions (with domain  ${\cal L}(\V)$), these functions are continuous. Hence the $k$th component of $\eta(T)$, being the difference of $S_k(\eta(T))$ and $S_{k-1}(\eta(T))$, is also continuous in $T$. This means that $\eta(T)$ is continuous in  $T$.
\\
$(f)$ Suppose $T_1(x)\prec T_2(x)$ for all $x\in \V$. Then, by definition $\lambda(T_1(x))\prec \lambda(T_2(x))$ for all $x\in \V$. Hence, 
$|\lambda(T_1(x))|\wprec |\lambda(T_2(x))|$, or equivalently, $\lambda(|T_1(x)|)\wprec \lambda(|T_2(x)|)$ for all $x$. Then, the pointwise inequality
$$\lambda(|T_1(x)|)\wprec |\lambda(|T_2(x)|)\wprec \eta(T_2)*\lambda(|x|)$$
implies that $\eta(T_1)\wprec \eta(T_2)$.
 \end{proof}

\gap

\noindent{\bf Remark 4.} In this remark, it is convenient to let $p,r,s$ denote real numbers. 
For $r\in [1,\infty]$ and $u\in \Rn$, $||u||_r$ is the usual $r$-norm of $u$ in $\Rn$. For any $x\in \V$, we define the
 corresponding {\it spectral norm}
$$||x||_r := ||\lambda(x)||_r,$$
which is $\big [\sum_{i=1}^{n} |\lambda_i(x)|^r\big ]^{1/r}$ when $1\leq r<\infty$ and
$\mbox{max}_{1\leq i\leq n} |\lambda_i(x)|$ when $r=\infty$. 
\\
Given $r,s\in [1,\infty]$ and $T\in {\cal L}(\V)$, we define the norm of $T$ from $(\V, ||\cdot||_r)$ to $(\V, ||\cdot||_s)$ by
$$||T||_{r\rightarrow s}:=\underset{0\neq x\in \V}{\sup} \frac{||T(x)||_s}{||x||_r}.$$
Now, suppose $p,r, s\in [1,\infty]$ with $\frac{1}{p}=\frac{1}{r}+\frac{1}{s}$. Then, the following inequality holds for any $T\in {\cal L}(\V)$ and $x\in \V$:
\begin{equation}\label{norm inequality for T}
||T(x)||_p\leq ||\eta(T)||_r\,||x||_s.
\end{equation}
 To see this, we follow the argument given in Theorem 5.1 of \cite{tao-jeong-gowda}. Starting from the inequality $\lambda(|T(x)|)\underset{w}{\prec} \eta(T)*\lambda(|x|)$, 
(\ref{norm inequality for T}) is easily seen when $p=\infty$ and $p=1$. For $1<p<\infty$, we use the fact that the function $\phi:t\mapsto t^p$ is an increasing convex function on $[0,\infty)$ to get (\cite{bhatia}, Exercise II.3.2)
$$||T(x)||_p\leq ||\eta(T)*\lambda(|x|)||_p.$$
An application of (classical) generalized H\"{o}lder's inequality gives (\ref{norm inequality for T}). Additionally, as  in Theorem 5.1 of \cite{tao-jeong-gowda}, we can prove the following:
$$||T||_{r\rightarrow s}\leq \left \{
\begin{array}{lll}
||\eta(T)||_\infty & if & r\leq s,\\
||\eta(T)||_{\frac{rs}{r-s}} & if &  s <r.
\end{array}
\right .
$$
To see a special case, suppose $T$ is positive and self-adjoint. In this case, $\eta(T)=\lambda(T(e))$ and so,
$$||T||_{r\rightarrow s}\leq \left \{
\begin{array}{lll}
||T(e)||_\infty & if & r\leq s,\\
||T(e)||_{\frac{rs}{r-s}} & if &  s <r.
\end{array}
\right .
$$
In some special cases, equality holds, see  e.g., Theorem 5.1 in \cite{tao-jeong-gowda}.
\section{Examples}

In this section, we present some examples. 
\\

\noindent{\bf Example 1} ({\it Lyapunov transformation $L_a$ and the quadratic representation $P_a$})
Recall that for any $a\in \V$, $L_a$ is defined by $L_a(x)=a\circ x$. As mentioned in the Introduction, we have $\lambda(|L_a(x)|)\wprec \lambda(|a|)*\lambda(|x|)$ for all $x\in \V$. (Note that this result does not come from any of our results above.) 
By our main theorem, we see that $\eta(L_a)\wprec \lambda(|a|).$  However, by putting $x=e$ in the inequality $\lambda(|L_a(x)|)\wprec \eta(L_a)*\lambda(|x|)$ we see that 
$\lambda(|a|)\wprec \eta(L_a)$. Hence, $\eta(L_a)=\lambda(|a|)$.  Now, $P_a$, defined by $P_a(x):=2\,a\circ (a\circ x)-a^2\circ x$,  is positive and self-adjoint; hence, from Corollary \ref{corollary for positive maps}, we have $\eta(P_a)=\lambda(P_a(e))=\lambda(a^2)$.\\

\noindent{\bf Example 2} ({\it Doubly stochastic maps}) \\
Suppose  $T$ is doubly stochastic so that $T$ is positive and $T(e)=T^*(e)=e$. In this case, from Corollary \ref{corollary for positive maps},  $\lambda(|T(x)|)\wprec \lambda(|x|)$ for all $x\in \V$ and $\eta(T)=\one$. As noted earlier, $\lambda(T(x))\prec \lambda(x)$ for all $x\in \V$. 
These apply when $T$ is a convex combination of algebra automorphisms \cite{gowda-positive and ds}. 
\\

\noindent{\bf Example 3} ({\it Doubly substochastic maps})\\
When $T$ is doubly substochastic, $T$ is positive with  $T(e)\leq e$ and  $T^*(e)\leq e$. So, 
$\eta(T)=\lambda(T(e))\underset{w}{\vee}\lambda(T^*(e))\wprec \lambda(e)= \one$ and 
$\lambda(|T(x)|)\underset{w}{\prec} \lambda(|x|)\,\,\mbox{for all}\,\, x\in \V,$
or equivalently, $\lambda(T(x))\underset{w}{\prec} \lambda(x)\,\,\mbox{for all}\,\, x\geq 0.$
\\

\noindent{\bf Example 4} ({\it Cone automorphisms})\\
Suppose $\V$ is a {\it simple} Euclidean Jordan algebra and $T\in \overline{\mbox{Aut}(\V_+)}$ (the closure of $\mbox{Aut}(\V_+)$ in ${\cal L}(\V)$). We claim that 
$$\eta(T)=\lambda(T(e)).$$
To see this, first suppose $T\in \mbox{Aut}(\V_+)$. Because $\V$ is assumed to be simple, we can write $ T=P_{a}\,\Phi$, where $\Phi$ is an algebra automorphism and $P_{a}$ is the quadratic representation of (some) $a>0$, see \cite{faraut-koranyi}, Page 56.  So,
by Remark  2, $\eta(T)=\lambda(T(e))$.  
Now the result for any $T\in \overline{\mbox{Aut}(\V_+)}$ follows from the  continuity of $\eta$ and $\lambda$. 
In the same setting, consider a linear map $S:\V\rightarrow \V$ which is a sum of a finite number of maps in $\overline{\mbox{Aut}(\V_+)}$.
Then, $S$ is a positive map and Corollary \ref{corollary for positive maps} can be applied. To illustrate these results, let $\V=\Hn$. Then, any $T\in \overline{\mbox{Aut}(\Hn_+)}$ is of the form $T(X)=AXA^*$ for some $n\times n$ complex square matrix $A$. In this setting, 
$$\eta(T)=\lambda(T(I))=\lambda(AA^*).$$
Now consider a {\it completely positive map} $S$ on $\Hn$, which is, by definition, a finite sum of the form  
$S(X):=\sum_{k=1}^{N}A_kXA_k^*$ with $A_k\in \mn$ for all $k$.  Letting $C:=S(I):=\sum_{k=1}^{N}A_kA_k^*$ and $D:=S^*(I):=\sum_{k=1}^{N}A^*_kA_k$, we have
$$\eta(S)\wprec \lambda(C)\underset{w}{\vee} \lambda(D)\wprec \max\{\lambda(C),\lambda(D)\}.$$

\noindent{\bf Example 5} ({\it ${\bf Z}$ and Lyapunov-like transformations})\\
We say that a  linear map $L:\V\rightarrow \V$ is a ${\bf Z}$-transformation \cite{gowda-tao-z} if  
$$\left [\,x,y\geq 0\,\,\mbox{and}\,\,\langle x,y\rangle=0\,\right ]\Rightarrow \langle L(x),y\rangle \leq 0.$$ 
It is said to be {\it Lyapunov-like} if the inequality on the right becomes an equality.\\
Such maps appear in dynamical systems theory. If a  ${\bf Z}$-transformation $L$ is also positive stable (meaning that  all eigenvalues of $L$ have positive real parts), 
then it is known that $L^{-1}$ is a positive map on $\V$ \cite{gowda-tao-z}. 
In this case, Corollary \ref{corollary for positive maps} is applicable to $L^{-1}$. To see an important special case, suppose $A$ is an $n\times n$  positive stable complex matrix. Then, $L_A$, defined on $\Hn$ by $L_A(X):=AX+XA^*$ is Lyapunov-like and positive stable. Let
$$C:=L^{-1}_{A}(I)\quad\mbox{and}\quad D:=L^{-1}_{A^*}(I).$$
Then, 
$$\eta(L^{-1}_A)=\lambda(C)\underset{w}{\vee}\lambda(D)\wprec \max\{\lambda(C),\lambda(D)\}.$$
We note that for any $X\in \Hn$, $L^{-1}_{A}(X)$ has an integral representation
$$L^{-1}(X)=\int_{0}^{\infty} e^{tA}Xe^{-tA^*}dt,$$
with a similar representation for $L^{-1}_{A^*}(X)$. These will give us integral representations for $C$ and $D$, but it is unclear how to represent $\eta(T)$ either in an integral form or in a closed form.   
\\

\noindent{\bf Example 6} ({\it L\"{o}wner maps})\\ 
Given a function $\phi:\R\rightarrow \R$, consider  the corresponding L\"{o}wner  map  $\phi$ defined on $\V$ (see Section 2). 
Motivated by the inequality (\ref{main inequality2}), we ask if the absolute value function can be replaced by $\phi$.
Keeping close to the properties of the absolute value function, we say that 
a function $\phi : \R \to \R$ is {\it sublinear} if
\begin{enumerate}
\item $\phi(\mu t) = \mu\phi(t)$ for all $\mu \geq 0$ and $t \in \R$;
\item  $\phi(t+s) \leq \phi(t) + \phi(s)$ for all $t,s\in \R$.
\end{enumerate}

It is easy to see that sublinear functions on $\R$ are of the form $\phi(t)=\alpha\,t$ for $t\geq 0$ and $\phi(t)=\beta\,t$ for $t\leq 0$, where (constants) $\alpha,\beta\in \R$ satisfy
$\beta\leq \alpha$. Among these, we consider ones that are nonnegative (that is, $\phi(t)\geq 0$ for all $t$). Examples include 
$$\phi(t)=|t|, \,\,\phi(t)=\max\{t,0\},\,\,\mbox{and}\,\,\phi(t)= \max\{-t,0\}.$$  
Now {\it suppose $\phi$ is sublinear and nonnegative,} and consider the corresponding L\"{o}wner  map. Then for any positive linear map $T$ we have 
$$\lambda\Big (\phi(T(x))\Big )\underset{w}{\prec}\lambda(T(\phi(x))\wprec \eta(T)*\lambda(\phi(x))\,\, \mbox{for all}\,\,x\in \V.$$
Here, the first inequality follows from Lemma 3.6 in \cite{tao-jeong-gowda} and the second inequality follows from 
Remark 1, Item $(ii)$. 
\\

\noindent{\bf Example 7}
For any $a,\, b \in \V$, consider the map 
        \[ P_{a, b}: = L_a L_b + L_b L_a - L_{a \circ b}. \]
        Clearly, $P_{a,a}=P_a$. Now, for any $a > 0$, $0 \leq t \leq 1$, it has been shown in \cite{gowda-schur} that
        $$ P_{\sqrt{a}}(x) \prec P_{a^t, a^{1-t}}(x) \prec L_a(x)\,\,\mbox{for all}\,\,x\in \V.$$
        Hence, by the `isotonicity' property  of $\eta$ (Item $(f)$ in Theorem \ref{theorem on eta}), we have
        \[ \eta(P_{\sqrt{a}}) \underset{w}{\prec} \eta(P_{a^t, a^{1-t}}) \underset{w}{\prec} \eta(L_a). \]
        Since $\eta(P_{\sqrt{a}}) =  \lambda(a)=\eta(L_a)$ (note $a>0$), the above inequality reduces to $\lambda(a) \underset{w}{\prec} \eta(P_{a^t, a^{1-t}}) \underset{w}{\prec} \lambda(a)$. It follows that 
$$\eta(P_{a^t, a^{1-t}}) = \lambda(a)\quad (a > 0,\,0 \leq t \leq 1).$$ 
It would be interesting to compute $\eta(P_{a, b})$ for general $a,\, b \in \V$.
\\

\noindent{\bf Example 8} ({\it Schur product induced maps})\\
Consider a fixed Jordan frame $\{e_1,e_2,\ldots, e_n\}$ in $\V$. This induces the Peirce decomposition (\cite{faraut-koranyi}, Theorem IV.2.1) $\V=\sum_{1\leq i\leq j\leq n} \V_{ij}$, and for any $x\in \V$,
$x=\sum_{i\leq j} x_{ij}$, where $x_{ij}\in \V_{ij}$.
Now, for any given $A=[a_{ij}]\in \Sn$, we define the Schur product $A\bullet x:=\sum_{i\leq j} a_{ij}x_{ij}.$ 
Properties of the Schur product and the induced transformation $D_A:x\mapsto A\bullet x$ are studied in \cite{gowda-sznajder-tao}.  The Lyapunov transformation $L_a$ and the quadratic representation $P_a$ are special cases. Further special cases are described below. 
\\
$(i)$ Suppose $A\in \Sn$ is positive semidefinite and consider the map $D_A:\V\rightarrow \V$ defined by $D_A(x):=A\bullet x$.
Then, $D_A$ is a self-adjoint positive linear map. So, 
$$\eta(D_A)=\lambda(D_A(e))=(\mbox{diag}\,A)^\downarrow.$$
In what follows, we provide an estimate for $\eta(D_A)$, where $A$ is not necessarily positive semidefinite. We write $A = A^{+} - A^{-}$, where $A^{+}$ and $A^{-}$ are positive semidefinite. 
By Theorem 4.1, we have 
        \begin{align*}
                \eta(D_A) & = \eta(D_{A^{+}} - D_{A^{-}}) \\
                & \underset{w}{\prec} \eta(D_{A^{+}}) + \eta(D_{A^{-}}) \\
                & = (\diag(A^{+}))^\downarrow  + (\diag(A^{-}))^\downarrow.
        \end{align*}
On the other hand, $\abs{\diag(A)}^\downarrow=\lambda(|D_A(e)|)\wprec \eta(D_A)*\lambda(e)=\eta(D_A).$
        Hence, we have 
$$\abs{\diag(A)}^\downarrow \underset{w}{\prec} \eta(D_A) \underset{w}{\prec} (\diag(A^{+}))^\downarrow + (\diag(A^{-}))^\downarrow$$ for any symmetric matrix $A$.
\\

$(ii)$ Now suppose $A=[a_{ij}], B=[b_{ij}]\in \Sn$ with $b_{ij}\neq 0$ for all $i,j$. Define the matrix $C:=\big [\frac{a_{ij}}{b_{ij}}\big ]\in \Sn$. Then, 
$A\bullet x=C\bullet (B\bullet x)=D_C(B\bullet x).$ Hence,
$$\lambda(|A\bullet x|)\wprec \eta(D_C)*\lambda(|B\bullet x|)\quad (x\in \V).$$
In particular, when $C$ is positive semidefinite (in which case, $D_C$ is a self-adjoint, positive map),
$$\lambda(|A\bullet x|)\wprec (\mbox{diag}\,C)^\downarrow *\lambda(|B\bullet x|)\quad (x\in \V).$$
$(iii)$ Suppose $\lambda(A\bullet x)\prec \lambda(B\bullet x)$ for all $x\in \V$. Pointwise majorization results of this type have been recently studied in \cite{gowda-schur}.
Thanks to the isotonicity of $\eta$, the pointwise inequality $\lambda(A\bullet x)\prec \lambda(B\bullet x)$ implies that $\eta(D_A)\wprec \eta(D_B)$. Example 7 above is an illustration of this.

\section{Some open problems}
Motivated by our results and examples, we raise some open problems.\\

\noindent{\bf Problem 1.} Consider a linear map $T:\Hn\rightarrow \Hn$. This can be extended to a (complex) linear map $\wtt:\mn\rightarrow \mn$ by 
$$\wtt(X):=T(A)+i\,T(B),$$
where $A:=\frac{X+X^*}{2}$ and $B:=\frac{X-X^*}{2i}$ are in $\Hn$. Then, we have $\eta(T)$ coming from Theorem \ref{main theorem} and $\eta(\wtt)$ 
coming from the result of Bapat (mentioned in the Introduction). To emphasize the algebras involved and to differentiate them, let us write $\eta(T,\Hn)$ and $\eta(\wtt,\mn)$. Then,
$s\big(\wtt(X)\big)\wprec \eta\big(\wtt,\mn\big)*s(X)$ for all $X\in \mn$ and $\lambda(|T(X)|)\wprec \eta(T,\Hn)*\lambda(|X|)$ for all $X\in \Hn$. From Theorem \ref{main theorem}  we see that 
$$\eta(T,\Hn)\wprec \eta(\wtt,\mn).$$
We now ask if (or under what conditions) the equality holds in the above. 
\\

\noindent{\bf Problem 2.} For a matrix $A\in \mn$, consider the Lyapunov transformation $L_A$ on $\Hn$. Computing the norms of $L_A$ and its inverse (whenever defined) 
relative to spectral $p$-norms is still an open problem 
\cite{feng et al, bhatia-lyapunov}. A related problem could be the description of  $\eta(L_A)$; see Example 1 (for Hermitian $A$).\\

\noindent{\bf Problem 3.} Given $A\in \Sn$ a fixed Jordan frame in $\V$, consider the Schur product induced map $D_A$ (as in Example 8). Is there a description of 
$\eta(D_A)$?


\end{document}